\documentclass[11pt]{article}
\usepackage{amssymb}
\usepackage{amsthm}
\usepackage{amsmath, mathrsfs}
\usepackage{amsfonts}
\newcommand{\be}{\begin{equation}}
\newcommand{\ee}{\end{equation}}
\newcommand{\bea}{\begin{eqnarray*}}
\newcommand{\eea}{\end{eqnarray*}}
\newcommand{\ba}{\begin{array}}
\newcommand{\ea}{\end{array}}
\newcommand{\bi}{\begin{itemize}}
\newcommand{\ei}{\end{itemize}}
\newcommand{\bc}{\begin{center}}
\newcommand{\ec}{\end{center}}
\newcommand{\bfr}{\begin{flushright}}
\newcommand{\efr}{\end{flushright}}

\newtheorem{Pa}{Paper}[section]
\newtheorem{theorem}[Pa]{{\bf Theorem}}
\newtheorem{La}[Pa]{{\bf Lemma}}
\newtheorem{Dn}[Pa]{{\bf Definition}}
\newtheorem{Cy}[Pa]{{\bf Corollary}}

\newtheorem{Rk}[Pa]{{\bf Remark}}
\newtheorem{proposition}[Pa]{{\bf Proposition}}

\begin{document}

\title{\bf Polynomial approximation in slice regular Fock spaces}
\author{
Kamal Diki \thanks{Marie Sklodowska-Curie fellow of the Istituto Nazionale di Alta Matematica} \\
Dipartimento di Matematica\\
Politecnico di Milano\\
Via Bonardi 9\\
20133 Milano, Italy\\
kamal.diki@polimi.it
\and Sorin G. Gal\\
University of Oradea\\
Department of Mathematics\\ and Computer Science\\
Str. Universitatii Nr. 1\\
410087 Oradea,
Romania \\
galso@uoradea.ro
\and Irene Sabadini\\
Dipartimento di Matematica\\
Politecnico di Milano\\
Via Bonardi 9\\
20133 Milano, Italy\\
irene.sabadini@polimi.it}

\date{}
\maketitle

\begin{abstract}
The main purpose of this paper is to prove some density results of polynomials in Fock spaces of slice regular functions. The spaces can be of two different kinds since they are equipped with different inner products and contain different functions. We treat both the cases, providing several results, some of them based on constructive methods which make use of the Taylor expansion and of the convolution polynomials. We also prove quantitative estimates in terms of higher order moduli of smoothness and in terms of the best approximation quantities.
\end{abstract}



\section{Introduction}

Fock spaces have been introduced in quantum mechanics via tensor products to describe the quantum states space of variables belonging to a same Hilbert space. Then, it was realized that this description correspond in fact to the Segal-Bargmann spaces, i.e. spaces of holomorphic functions in several variables which are square integrable with respect to a Gaussian measure. These spaces are important also in other settings among which we mention infinite dimensional analysis and free analysis; these latter spaces are related with the white noise space and with the theory of stochastic distributions, see \cite{dan}. For an account on the theory of Fock spaces one may consult for example the book \cite{Zhu_1}.

These spaces have been considered in some higher dimensional extensions of complex analysis, namely the analysis based on functions with values in a  Clifford algebra or, in particular, quaternions. For some recent works, we refer the reader to e.g. \cite{ppss} which is the framework of monogenic functions, to \cite{AlpayColomboSabadiniSalomon2014, DG} in the framework of slice hyperholomorphic functions and to \cite{kmnq}, which makes use of slice hyperholomorphic functions and in which the authors point out the link with the study of quantum systems with internal, discrete degrees of freedom corresponding to nonzero spins. The class of slice hyperholomorphic functions, see \cite{CSS,ColomboSabadiniStruppa2016,gss} has attracted interest in the past decade for its various applications especially in operator theory. One of its features is that it contains power series (despite what happens for other theories of hypercomplex variables) thus it is natural to consider functions which are "entire" in this class and, in particular, Fock spaces.

This paper continues the study of slice hyperholomorphic Fock spaces over the quaternions started in \cite{AlpayColomboSabadiniSalomon2014} with the purpose of providing some approximation results, specifically our goal is to extend to this setting results on the density of polynomials in the complex case. We shall show that in this context one may define two types of Fock spaces, which are called of the first and of the second kind, and for which the approximation results require different techniques.

The plan of the paper is the following: Section 2 contains some preliminary material and the definition of Fock spaces of the first and second kind.
Section 3 deals with approximation results in Fock spaces of the first kind, while Section 4 considers the case of Fock spaces of the second kind. It should be noted that the techniques used in these two sections are different. In the case of Fock spaces of the second kind we prove quantitative estimates in terms of the higher order moduli of smoothness and of the best approximation quantity. Finally, we discuss type and order of functions in the Fock spaces of the second kind.

\section{Preliminary results}
In this section, we provide the ground to show our main results that are presented.

The noncommutative field $\mathbb{H}$ of quaternions consists of elements of the form
$q=x_0  + x_{1} i +x_{2} j + x_{3} k$, $x_i\in \mathbb{R}$, $i=0,1, 2,3$,
where the imaginary units $i, j, k$ satisfy the relations
$$i^2=j^2=k^2=-1, \ ij=-ji=k, \ jk=-kj=i, \ ki=-ik=j.$$ The real number $x_0$ is called the  real part of $q$
while $ x_{1} i +x_{2} j + x_{3} k$ is called the imaginary part of $q$.
The norm of a quaternion $q$, is defined by
$|q|=\sqrt{x_0^2+x_1^2 +x_2^2+x_3^2}$.
\\
We denote by $\mathbb{B}_{r}$, the open ball in $\mathbb H$ of radius $r>0$, i.e.
 $$\mathbb{B}_{r}=\{q = x_0+i x_{1} + j x_{2} + k x_{3}, \mbox{ such that } x_0^2+x_{1}^{2}+x_{2}^{2}+x_{3}^{2}<r^{2}\}$$
while by  $\mathbb{S}$ we denote the unit sphere of purely imaginary quaternion, i.e.
$$\mathbb{S}=\{q = i x_{1} + j x_{2} + k x_{3}, \mbox{ such that } x_{1}^{2}+x_{2}^{2}+x_{3}^{2}=1\}.$$

Note that if $I\in \mathbb{S}$, then $I^{2}=-1$.  For any fixed $I\in\mathbb{S}$ we define $\mathbb{C}_I:=\{x+Iy; \ |\ x,y\in\mathbb{R}\}$, which can be identified with a complex plane.

Any non real quaternion $q$ is uniquely associated to the element $I_q\in\mathbb{S}$
given by $I_q:=( i x_{1} + j x_{2} + k x_{3})/|  i x_{1} + j x_{2} + k x_{3}|$ and $q$ belongs to the complex plane $\mathbb{C}_{I_q}$.

 We note that the real axis belongs to $\mathbb{C}_I$ for any $I\in\mathbb{S}$, so that a real quaternion is in fact associated with any imaginary unit $I$.
Furthermore, we have $\mathbb{H}=\cup_{I\in\mathbb{S}} \mathbb{C}_I$.

In the sequel we need the notion of convolution operators of a quaternion variable, and to introduce them we need a suitable notion of exponential function of quaternion variable. For every $I\in \mathbb{S}$, we consider
the following known definition for the exponential:
$e^{I t}=\cos(t)+ I \sin(t),\, \, t\in \mathbb{R}$, see \cite{GHS}.
The Euler's kind formula holds : $(\cos(t)+ I \sin(t))^{k}=\cos(k t)+ I \sin(k t)$, and therefore we have $(e^{I t})^{k}=e^{ I k t}$.

For any $q\in \mathbb{H}\setminus \mathbb{R}$, let
$r:=|q|$; then, see \cite{GHS}, there exists a unique $a\in (0, \pi)$ such that $\cos(a):=\frac{x_{1}}{r}$ and a unique $I_{q}\in \mathbb{S}$, such that
$$q=r e^{I_{q} a}, \, \mbox{ with } I_{q}=i y + j v + k s,\, y=\frac{x_{2}}{r \sin(a)},\, v=\frac{x_{3}}{r \sin(a)}, \, s=\frac{x_{4}}{r \sin(a)}.$$
Now, for $q\in \mathbb{R}$ we set $a=0$; if $q>0$ and $a=\pi$ if $q<0$, and we choose as $I_{q}$ an arbitrary fixed $I\in \mathbb{S}$.
So that if $q\in \mathbb{R}\setminus \{0\}$, then we can write $q=|q|(\cos(a)+I\sin(a))$, for a non unique $I$.
The above expression is called the trigonometric form of the quaternion number $q\not=0$.  For $q=0$, exactly as in the complex case, we do not have a trigonometric form.

In this paper we will work in the framework of  the so called slice regular functions of a quaternion variable. We define them below, also proving some main properties. More information on these functions and on their applications, can be found in \cite{CSS}, \cite{gss} and in the references therein.
\begin{Dn}\label{slice regular}
If $U$ is an open set of
$\mathbb{H}$, then a real differentiable
function $f:U \to \mathbb{H}$ is called (left) slice regular if,
for every $I \in \mathbb{S}$, its restriction $f_I$ to $U_I=U\cap\mathbb C_I$ satisfies
$$
\overline{\partial}_If(x+Iy)=\frac{1}{2}\Big(\frac{\partial}{\partial x}
+I\frac{\partial}{\partial y}\Big)f_I(x+Iy)=0.
$$
 The set of all slice regular functions on $U$ will be denoted by $\mathcal{SR}(U)$. If $U=\mathbb{H}$ then $\mathcal{SR}(\mathbb{H})$ will be called as the space of entire slice regular functions.
\end{Dn}
One of the most important properties of the class of slice regular functions, is that it contains converging power series (and polynomials) of the variable $q$ with the quaternionic coefficients written on the right.
Specifically, the following result holds.

\begin{theorem}\label{thm1} (see, e.g., Theorem 1.9 in \cite{gss}) Let $\mathbb{B}_{r}=\{q\in \mathbb{H} ; |q| < r\}$ and  $f:\ \mathbb{B}_{r}\to \mathbb{H}$. Then $f\in \mathcal{SR}(\mathbb{B}_{r})$ if and only if $f$ has a series representation of the form
$$f(q)=\sum_{n=0}^{\infty}q^{n}\frac{1}{n !}\cdot \frac{\partial^{n} f}{\partial x^{n}}(0),$$
uniformly convergent on compact sets in $\mathbb{B}_{r}$.
\end{theorem}

Two important properties of slice regular functions that will be very useful in the sequel are, the Splitting Lemma and the Representation Formula.

\begin{La}\label{split}(see, e.g., Lemma 1.3 in \cite{gss}) Let $f$ be a slice regular function on a domain $\Omega$. Then, for every $I,J\in\mathbb{S}$ with $I\perp J$, there exist two holomorphic functions $F,G:\Omega_{I}\longrightarrow{\mathbb{C}_I}$ such that for all $z=x+Iy\in{\Omega_I}$, we have
$$f_I(z)=F(z)+G(z)J,$$
 where $\Omega_I=\Omega\cap{\mathbb{C}_I}$ and $\mathbb{C}_I=\mathbb{R}+\mathbb{R}I.$
\end{La}
Note that the representation formula for slice regular functions holds on axially symmetric slice domains. Namely, we have
\begin{Dn}
A domain $\Omega\subset \mathbb{H}$ is said to be a slice domain (or just $s$-domain) if  $\Omega\cap{\mathbb{R}}$ is nonempty and for all $I\in{\mathbb{S}}$, the set $\Omega_I:=\Omega\cap{\mathbb{C}_I}$ is a domain of the complex plane $\mathbb{C}_I$.
If moreover, for every $q=x+Iy\in{\Omega}$, the whole sphere $x+y\mathbb{S}:=\lbrace{x+Jy; \, J\in{\mathbb{S}}}\rbrace$
is contained in $\Omega$, we say that  $\Omega$ is an axially symmetric slice domain.
\end{Dn}

\begin{theorem}\label{repform}(see, e.g., Corollary 1.16 in \cite{gss})
Let $\Omega$ be an axially symmetric slice domain and $f\in{\mathcal{SR}(\Omega)}$. Then, for any $I,J\in{\mathbb{S}}$, we have the formula
$$
f(x+Jy)= \frac{1}{2}(1-JI)f_I(x+Iy)+ \frac{1}{2}(1+JI)f_I(x-Iy)
$$
for all $q=x+Jy\in{\Omega}$.
\end{theorem}

According to \cite{CGLSS}, in the framework of slice regular functions, can be considered two kinds of (weighted) Bergman spaces. In the recent paper
\cite{gasa3}, the properties of density for quaternionic polynomials in these kinds of spaces were obtained. Let us mention here that in the complex case, convolution polynomials were used to obtain constructive approximation results in complex Bergman spaces, see \cite{gal1}.

In this paper we continue this type of study for the quaternionic slice regular functions in the so-called Fock spaces.
To this purpose, it is worthwhile to provide some background on Fock spaces in the complex case (see, e.g., \cite{Zhu_1}).

\begin{Dn} (see, e.g., \cite{Zhu_1}, p. 36) Let $0 < p < \infty$ and $\alpha>0$. The Fock space $F^{p}_{\alpha}(\mathbb{C})$ is defined as the space of all entire functions in $\mathbb{C}$ with the property that $\displaystyle\frac{\alpha p}{2\pi}\int_{\mathbb{C}}\left |f(z)e^{-\alpha |z|^{2}|/2}\right |^{p}d A(z)<+\infty$, where $d A(z)=d x d y=r d r d \theta$, $z=x+i y=re^{i \theta}$, is the area measure in the complex plane.
\end{Dn}
\begin{Rk}{\rm Endowed with
$$\|f\|_{p, \alpha}^p=\displaystyle\frac{\alpha p}{2\pi} \int_{\mathbb{C}}\left |f(z)e^{-\alpha |z|^{2}/2}\right |^{p}d A(z),$$
it is known (see, e.g., \cite{Zhu_1}, p. 36) that $F_{\alpha}^{p}$ is a Banach space for $1\le p <\infty$, and a complete metric space for
$\|\cdot\|^{p}_{p, \alpha}$ with $0<p<1$. Also, if $p=+\infty$, then  endowed with $\|f\|_{\infty, \alpha}=esssup\{|f(z)|e^{-\alpha |z|^{2}|/2} ; z\in \mathbb{C}\}$, $F^{\infty}_{\alpha}$ is a Banach space.}
\end{Rk}
\begin{Rk} {\rm Concerning the approximation by polynomials in Fock spaces, qualitative results without any quantitative estimates were obtained. For any $0<p<\infty$, and $f\in F_{\alpha}^{p}$, there exists a polynomial sequence $(P_{n})_{n\in \mathbb{N}}$ such that $\lim_{n\to \infty}\|f-P_{n}\|_{p, \alpha}=0$ (see, e.g., Proposition 2.9, p. 38 in \cite{Zhu_1}). The proof of the result is not constructive and consists in two steps : at step 1, one approximates $f(z)$ by its dilations $f(r z)$ with $r\to 1^{-}$ and at step 2 one approximates each $f_{r}$ by its attached Taylor polynomials. If $1<p<\infty$, then one can construct $P_{n}$ as the Taylor polynomials attached to $f$ (see, e.g., Exercise 5, p. 89 in \cite{Zhu_1}) but if $0<p\le 1$, then there exists $f\in F_{\alpha}^{p}$ which cannot be approximated by its associated Taylor polynomials (see, e.g., Exercise 6, p. 89 in \cite{Zhu_1}). However, if $f\in F_{\alpha}^{\infty}$ is such that $\lim_{z\to \infty}f(z)e^{\alpha |z|^{2}/2}=0$, then
$f$ can be approximated by polynomials in the norm $\|\cdot \|_{\infty, \alpha}$ (see, e.g., Exercise 8, p. 89 in \cite{Zhu_1}).}
\end{Rk}
We can now introduce the quaternionic counterpart of Fock spaces, beginning with the following definition that, in this generality,  has not been previously considered in the literature.
\begin{Dn}
Let $0<p<+\infty$ and $0<\alpha<+\infty$. The Fock space of the first kind $\mathcal{F}^{p}_{\alpha}(\mathbb{H})$ is defined as the space of entire slice regular functions $f\in \mathcal{SR}(\mathbb{H})$, such that
$$\|f\|_{p, \alpha}^p:=\displaystyle\left(\frac{\alpha p}{2\pi}\right)^2\int_{\mathbb{H}}|f(q)|^{p}(e^{-\alpha|q|^{2}/2})^{p}d m(q)<+\infty,$$
where $d m(q)$ represents the Lebesgue volume element in $\mathbb{R}^{4}$.
\end{Dn}
\begin{Rk}{\rm
  Using standard techniques, like in the complex case, one can prove that for $1\le p <+\infty$, $\|\cdot\|_{p, \alpha}$ is a norm, while for $0<p<1$, $\|f-g\|_{p, \alpha}^{p}$ is a quasi-norm.
}
\end{Rk}
Before to define the Fock spaces of the second kind, we give the following:
\begin{Dn}
For $I\in \mathbb{S}$, $0<\alpha <+\infty$ and $0<p<+\infty$, let us denote
$$
 \|f\|_{p, \alpha, I}^p=\displaystyle\frac{\alpha p}{2\pi}\int_{\mathbb{C}_{I}}|f(q)|^{p}(e^{-\alpha |q|^{2}/2})^{p}d m_{I}(q),
$$
 with $d m_{I}(q)$ representing the area measure on $\mathbb{C}_{I}$.
\\
The space of all entire functions $f$ satisfying $\|f\|_{p, \alpha, I}<+\infty$ will be denoted with $\mathcal{F}^{p}_{\alpha, I}(\mathbb{H})$.
\end{Dn}
Now, we are in position to introduce the following:
\begin{Dn}\label{Berg_2}
The Fock space of the second kind $\mathcal{F}^{\alpha, p}_{Slice}(\mathbb{H})$ consists in all $f\in \mathcal{SR}(\mathbb{H})$ having the property that for some $I\in \mathbb{S}$ it follows $f\in \mathcal{F}^{p}_{\alpha, I}(\mathbb{H})$. In order to make the norm independent of the choice of the imaginary unit, we set $$\|f\|_{\mathcal{F}^{\alpha, p}_{Slice}(\mathbb{H})}=\underset{I\in\mathbb{S}} \sup \|f\|_{p, \alpha, I}.$$
\end{Dn}

\section{Polynomial approximation in Fock spaces of the first kind}
In the sequel, we consider the quaternionic Fock spaces of the first kind $\mathcal{F}^{p}_{\alpha}(\mathbb{H})$ introduced in Section 2. First, we start by proving the following estimate:

\begin{La} \label{EEH}
Let $f\in\mathcal{F}^{p}_{\alpha}(\mathbb{H})$. Then, there exists a constant $c>0$ such that for all $q\in\mathbb{H}$, we have $$|f(q)|\leq ce^{\frac{\alpha}{2}|q|^2}\|f\|_{p, \alpha},$$
where $c=4\displaystyle\left(\frac{2\pi}{\alpha p}\right)^\frac{1}{p}$.
\end{La}
\begin{proof}
Let $I\in \mathbb{S}$, since $f$ is slice regular on $\mathbb{H}$ , then making use of the Splitting Lemma  we have that for all $z\in\mathbb{C}_I$,
$$f_I(z)=F(z)+G(z)J,$$ where $J\in\mathbb{S}$ is orthogonal to $I$, and $F$, $G$ are two holomorphic functions on the slice $\mathbb{C}_I$. Note that since $f\in\mathcal{F}^{p}_{\alpha}(\mathbb{H})$ it is easy to show that $F$ and $G$ belong to the classical Fock space $\mathcal{F}^{p}_{\alpha}(\mathbb{C}_I)$. Thus, by the classical complex analysis the following estimates are satisfied for any $z\in\mathbb{C}_I$ $$|F(z)|\leq e^{\frac{\alpha}{2}|z|^2}\Vert{F}\Vert_{\mathcal{F}^p_\alpha(\mathbb{C}_I)} \text{ and } |G(z)|\leq e^{\frac{\alpha}{2}|z|^2}\Vert{G}\Vert_{\mathcal{F}^p_\alpha(\mathbb{C}_I)}.$$ Then, \[
\begin{split}
|f(z)|&\le |F(z)|+|G(z)|\\
&\le
e^{\frac{\alpha}{2}|z|^2}(\Vert{F}\Vert_{\mathcal{F}^p_\alpha(\mathbb{C}_I)}+\Vert{G}\Vert_{\mathcal{F}^p_\alpha(\mathbb{C}_I)}).
\end{split}
\]
However, since $|F(z)|\le |f(z)|$ for any $z\in\mathbb{C}_I$, we have \[ \begin{split}
 \Vert{F}\Vert_{\mathcal{F}^p_\alpha(\mathbb{C}_I)}^p& = \displaystyle\frac{\alpha p}{2\pi}\int_{\mathbb{C}_{I}}|F(z)|^{p}(e^{-\alpha |z|^{2}/2})^{p}d m_{I}(z)\\
 &\le \displaystyle\frac{\alpha p}{2\pi}\int_{\mathbb{C}_I}|f(z)|^{p}(e^{-\alpha |z|^{2}/2})^{p}d m_I(z)\\
&\le \displaystyle\frac{\alpha p}{2\pi}\int_{\mathbb{H}}|f(q)|^{p}(e^{-\alpha |q|^{2}/2})^{p}d m(q)\\
& = \displaystyle\frac{2\pi}{\alpha p} \|f\|_{p, \alpha}^p .
\end{split}
\]

By similar arguments we get also $ \Vert{G}\Vert_{\mathcal{F}^p_\alpha(\mathbb{C}_I)} \le \displaystyle\left(\frac{2\pi}{\alpha p}\right)^\frac{1}{p}\|f\|_{p, \alpha}$. So, for any $z\in\mathbb{C}_I$ we have the following estimate
$$|f(z)|\le 2\displaystyle\left(\frac{2\pi}{\alpha p}\right)^\frac{1}{p}e^{\frac{\alpha}{2}|z|^2}\|f\|_{p, \alpha}.$$

Finally, for $q=x+Jy\in\mathbb{H}$ by the Representation Formula  we have $$f(q)=\displaystyle \frac{1}{2}\left[f(z)+f(\overline{z})\right]+J\frac{I}{2}\left[f(\overline{z})-f(z)\right]; z=x+Iy\in\mathbb{C}_I$$

Thus,
$$|f(q)|\leq |f(z)|+|f(\overline{z})|.$$
Hence, the last inequality combined with the estimate on $\mathbb{C}_I$ give $$|f(q)|\le 4\displaystyle\left(\frac{2\pi}{\alpha p}\right)^\frac{1}{p}e^{\frac{\alpha}{2}|q|^2}\|f\|_{p, \alpha},$$ for all $q\in\mathbb{H}$.
\end{proof}

Now, we present the main result in this setting concerning the polynomial approximation.
\begin{theorem}\label{thm_first_kind} Let $\alpha > 0$ and $0<p<\infty$. The set of all quaternionic polynomials is included in $\mathcal{F}^{p}_{\alpha}(\mathbb{H})$ and for every $f\in \mathcal{F}^{p}_{\alpha}(\mathbb{H})$, there exists a sequence of quaternionic polynomials $(p_{n})_{n\in \mathbb{N}}$ with $\|p_{n}-f\|_{p, \alpha}\to 0$ as $n\to +\infty$.
\end{theorem}
\begin{proof} First of all, we observe that any quaternionic polynomial belongs to $\mathcal{F}^{p}_{\alpha}(\mathbb{H})$. This follows easily from the fact that for any $k=0, 1, ....$, we have
$$\int_{\mathbb{H}}|q^{k}|^{p}(e^{-\alpha|q|^{2}/2})^{p}d m(q) <+\infty.$$

We then divide the proof in two steps.

{\bf Step 1.} Let $0<r<1$, $f\in \mathcal{F}^{p}_{\alpha}(\mathbb{H})$ and define $f_{r}(q)=f(r q)$. Evidently $f_r$ is an entire slice regular function.

Firstly, we will prove that $\lim_{r\to 1^{-}}\|f_r - f\|_{p, \alpha}=0$. We will reason similar to the complex case in the proof of Proposition 2.9, p. 38 in \cite{Zhu_1}, taking into account that $f:\mathbb{H}\to \mathbb{H}$ can be written componentwise as
$$f(q)=f_1(x_1, x_2, x_3, x_4)+i f_2(x_1, x_2, x_3, x_4)+j f_3(x_1, x_2, x_3, x_4)+ k f_4(x_1, x_2, x_3, x_4),$$
$q=x_1+i x_2 + j x_3 + k x_4$ and that applying the Lemma 3.17, p. 66 in \cite{HKZ} to $f$ is equivalent to apply it to each real-valued function of four real variables $f_{k}(x_1, x_2, x_3, x_4)$, $k=1, 2, 3, 4$.

By using the componentwise form, since $f$ is entire slice regular it follows that it is continuous on $\mathbb{H}$ and it is immediate that $\lim_{r\to 1^{-1}}f(rq)=f(q)$, for all $q\in \mathbb{H}$.

Now, for $f\in \mathcal{F}^{p}_{\alpha}(\mathbb{H})$,  changing the variable $r q=w$ and taking into account that as in the proof of Theorem 2.1 in \cite{gasa3}, we have $d m(q)=\frac{1}{r^{4}}d m(w)$, we obtain
\[
\begin{split}
\|f_{r}\|^{p}_{p, \alpha}&= \displaystyle\left(\frac{\alpha p}{2\pi}\right)^2\int_{\mathbb{H}}|f(r q) e^{-\alpha |q|^{2}/2}|^{p}d m(q)\\
&=\displaystyle\left(\frac{\alpha p}{2\pi}\right)^2\frac{1}{r^{4}}\cdot \int_{\mathbb{H}}|f(w)e^{-\alpha |w|^{2}/2}|^{p}\cdot e^{-p \alpha |w|^{2}(r^{-2}-1)/2} d m(w).
\end{split}
\]

Since for all $w\in \mathbb{H}$ and $0<r<1$ we have $e^{-p \alpha |w|^{2}(r^{-2}-1)/2}\le 1$, by applying the dominated  convergence theorem in the above mentioned Lemma 3.17 in \cite{HKZ}, we are lead to $\lim_{r\to 1^{-1}}\|f_r - f\|_{p, \alpha}=0$.

{\bf Step 2.} The proof is terminated if we can show that for every $r\in (0, 1)$, the function
$f_{r}$ can be approximated by some quaternionic polynomials in the norm topology of $\mathcal{F}^{p}_{\alpha}(\mathbb{H})$. To this end, let $0<r<1$ and $\alpha r^2<\beta< \alpha$. On one hand, note that $f_r$ is slice regular on $\mathbb{H}$. Moreover,  according to Lemma \ref{EEH} there exists $c>0$ such that for any $q\in\mathbb{H}$ we have $$|f_r(q)|=|f(rq)|\leq ce^{\frac{\alpha}{2}r^2|q|^2}\|f\|_{p, \alpha}.$$ Thus, since $\alpha r^2-\beta<0$ we get $$\displaystyle \int_\mathbb{H}|f_r(q)|^2e^{-\beta|q|^2}dm(q)\leq c^2\|f\|_{p, \alpha}^2\int_\mathbb{H}e^{(\alpha r^2-\beta)|q|^2}dm(q)< \infty. $$ In particular, this shows that $f_r$ belongs to $\mathcal{F}^{2}_{\beta}(\mathbb{H})$. Furthermore, since $\beta-\alpha<0$ we can see also that  $\mathcal{F}^{2}_{\beta}(\mathbb{H})$ is continuously embedded in $\mathcal{F}^{p}_{\alpha}(\mathbb{H})$. Indeed, for $h\in\mathcal{F}^{2}_{\beta}(\mathbb{H}),$ applying Lemma \ref{EEH} there exists  $C>0,$ such that for any $q\in\mathbb{H}$, we have  $$|h(q)|\leq C e^{\frac{\beta}{2}|q|^2}\|h\|_{2, \beta}.$$  Thus, $$\displaystyle \int_\mathbb{H}|h(q)|^pe^{-\frac{\alpha p}{2}|q|^2}dm(q)\leq C^p \|h\|_{2, \beta}^p\int_\mathbb{H}e^{\frac{(\beta-\alpha) p}{2}|q|^2}dm(q)< \infty.$$ Hence, this shows that  $\|h\|_{p, \alpha}\leq K\|h\|_{2, \beta}$ where $K=K(\alpha,\beta,p)>0$. On the other hand, it is clear that the quaternionic monomials $(q^n)_n$ are contained and generate any element of the quaternionic Hilbert space $\mathcal{F}^{2}_{\beta}(\mathbb{H})$ but they do not form an orthogonal basis of the Hilbert Fock space of the first kind. So, using the orthonormalization process we can obtain an orthonormal total family $(p_n)_n$ of quaternionic polynomials in $\mathcal{F}^{2}_{\beta}(\mathbb{H})$. Therefore, $f_r$ can be approximated by $(p_n)_n$ since $f_r\in\mathcal{F}^{2}_{\beta}(\mathbb{H})$. Moreover,there exists $K>0$ such that $$\|f_r-p_n\|_{p, \alpha}\leq K\|f_r-p_n\|_{2, \beta}.$$ Finally, the previous inequality shows that $f_{r}$ can be approximated by a sequence of quaternionic polynomials in the norm topology of $\mathcal{F}^{p}_{\alpha}(\mathbb{H})$. This ends the proof.

\end{proof}
\begin{Rk} {\rm The approximation in Fock spaces of the first kind is not based on the Taylor expansion since the quaternionic monomials do not form an orthogonal basis of the Hilbert Fock space of the first kind.}
\end{Rk}

\section{Polynomial approximation in Fock spaces of the second kind}

In this section we prove that polynomials are dense in Fock spaces of the second kind. Furthermore, we prove a result with quantitative estimates in terms of higher order moduli of smoothness and of the best approximation quantity.
\\
To this goal, we first need to prove some technical results. The following proposition has a rather standard proof that we write for the sake of completeness.

\begin{proposition}\label{equiva}  Let $p\geq 1$ (resp. $0<p<1$) and $\|\cdot \|_{p, \alpha, I}$ be the norm (resp. quasi-norm) in $\mathcal{F}^{ p}_{\alpha, I}(\mathbb{H})$. Then $\|\cdot \|_{p, \alpha, I}$ and $\|\cdot \|_{p, \alpha, J}$ are equivalent for any $I,J\in\mathbb S$.
\end{proposition}
\begin{proof}
From the representation formula we easily get
$$|f(x+y I)|\le |f(x+y J)|+|f(x-y J)|.$$
Then, by taking $| .|^p$ in the above formula,
and using the inequalities  $(a+b)^{p}\le 2^{p-1}(a^{p}+b^{p})$, if $1\le p <+\infty$,
and  $(a+b)^{p}\le a^{p}+b^{p}$, if $0<p<1$, for all $a, b \geq 0$ we obtain
$$|f(x+y I)|^{p}\le 2^{p-1} \left [|f(x+y J)|^{p}+|f(x-y J)|^{p}\right ], \mbox{ if } 1\le p <\infty$$
and
$$|f(x+y I)|^{p}\le [|f(x+y J)|^{p}+|f(x-y J)|^{p}], \mbox{ if } 0 < p < 1.$$
Multiplying both terms in the above inequalities with $e^{-p \alpha |x+y I|^{2}/2}$, integrating on $\mathbb{C}_{I}$ with respect to $d m_{I}(q)$, then multiplying the corresponding obtained inequality with $e^{-p \alpha |x+y J|^{2}/2}$, integrating on $\mathbb{C}_{J}$ with respect to $d m_{J}(q)$  and taking into account that $|x+y I|^{2}=|x+y J|^{2}=|x-y J|^{2}=x^{2}+y^{2}$, we obtain an inequality of the form
$\|f\|_{p, \alpha, I}\le C_{p}\|f\|_{p, \alpha, J}$,
with $C_{p}$ independent of $I$ and $J$.

Interchanging now $I$ with  $J$ and repeating the above reasonings, we get
the desired conclusion.
\end{proof}
\begin{Cy}
Given any $I,J\in\mathbb{S}$ the slice hyperholomorphic Fock spaces  $\mathcal{F}^{ p}_{\alpha, I}(\mathbb{H})$ and $\mathcal{F}^{ p}_{\alpha, J}(\mathbb{H})$  contain the same elements and have equivalent norms.
\end{Cy} \begin{Rk} {\rm The notion of Fock space of the second kind given in Definition \ref{Berg_2} is independent of the choice of the imaginary unit, and this justifies the notation $\mathcal{F}^{\alpha, p}_{Slice}(\mathbb{H})$.}
\end{Rk}
\begin{La} \label{GC}
Let $0<p<\infty, \alpha >0$ and $f\in\mathcal{F}^{\alpha, p}_{Slice}(\mathbb{H})$. Then, for any $q\in\mathbb{H}$ we have $$|f(q)|\leq 4e^{\frac{\alpha}{2}|q|^2}\Vert{f}\Vert_{\mathcal{F}^{\alpha, p}_{Slice}(\mathbb{H})}.$$
\end{La}
\begin{proof}
Let $f\in \mathcal{F}^{\alpha, p}_{Slice}(\mathbb{H})$ and let $I\in \mathbb{S}$. Then, choose  $J$ in $\mathbb{S}$ perpendicular to $I$. In particular, $f$ is slice regular on $\mathbb{H}$, then by the Splitting Lemma there exist $F,G:\mathbb{C}_I\longrightarrow \mathbb{C}_I$ two holomorphic functions such that we have $$f_I(z)=F(z)+G(z)J; \text{ } \forall z\in\mathbb{C}_I.$$ Then, we use similar arguments as in the Lemma \ref{EEH} to see that for any $z\in\mathbb{C}_I$, we have \[
\begin{split}
|f(z)|
&\le
e^{\frac{\alpha}{2}|z|^2}(\Vert{F}\Vert_{\mathcal{F}^p_\alpha(\mathbb{C}_I)}+\Vert{G}\Vert_{\mathcal{F}^p_\alpha(\mathbb{C}_I)}).
\end{split}
\]

However, note that $$\Vert{F}\Vert_{\mathcal{F}^p_\alpha(\mathbb{C}_I)}\leq \Vert{f}\Vert_{\mathcal{F}^{\alpha, p}_{Slice}(\mathbb{H})} \text{ and } \Vert{G}\Vert_{\mathcal{F}^p_\alpha(\mathbb{C}_I)}\leq \Vert{f}\Vert_{\mathcal{F}^{\alpha, p}_{Slice}(\mathbb{H})}.$$

Thus, for any $z\in\mathbb{C}_I$ we get $$|f(z)| \le 2 e^{\frac{\alpha}{2}|z|^2}\Vert{f}\Vert_{\mathcal{F}^{\alpha, p}_{Slice}(\mathbb{H})}.$$

Finally, we  apply the Representation Formula in order to prove the estimate for any $q\in\mathbb{H}$ and this completes the proof.
\end{proof}
Now, we can state and prove the first main result of this section.
\begin{theorem}\label{thm_second_kind} Let $0<p<+\infty$, $0<\alpha<+\infty$ and $f\in \mathcal{F}^{\alpha, p}_{Slice}(\mathbb{H})$. There exists a polynomial sequence $(P_{n})_{n\in \mathbb{N}}$ such that for any $I\in \mathbb{S}$ it follows $\|P_{n}-f\|_{p, \alpha, I}\to 0$ as $n\to +\infty$.
\end{theorem}
\begin{proof} The proof is divided  in two steps.

{\bf Step 1.} Fix $I_{0}\in \mathbb{S}$. For $0<r<1$, we define $f_{r}(q)=f(r q)$, $q\in \mathbb{H}$. By hypothesis, we have $f\in \mathcal{F}^{p}_{\alpha, I_{0}}(\mathbb{H})$, i.e. $f_{r}$ is an entire function on $\mathbb{C}_{I_{0}}$.

Firstly, we prove that $\lim_{r\to 1^{-}}\|f_{r}-f\|_{p, \alpha, I_{0}}=0$.
To this end, we note that since the restriction of $f_r$ is entire on $\mathbb{C}_{I_{0}}$, it is continuous, which evidently implies that pointwise we have $\lim_{r\to 1^{-}}f(r q)=f(q)$, for all $q\in \mathbb{C}_{I_{0}}$.

Now, for $f\in \mathcal{F}^{p}_{\alpha, I_{0}}(\mathbb{H})$, by setting $r q=w$ and taking into account that as in the proof of Theorem 2.1 in \cite{gasa3} (see also \cite{Bloom}), we have $d m_{I_{0}}(q)=\frac{1}{r^{2}}d m_{I_{0}}(w)$, we obtain

\[
\begin{split}
\|f_{r}\|^{p}_{p, \alpha, I_{0}}&= \frac{\alpha p}{2\pi}\int_{\mathbb{C}_{I_{0}}}|f(r q) e^{-\alpha |q|^{2}/2}|^{p}d m_{I_{0}}(q)\\
&=\frac{\alpha p}{2\pi}\frac{1}{r^{2}}\cdot \int_{\mathbb{C}_{I_{0}}}|f(w)e^{-\alpha |w|^{2}/2}|^{p}\cdot e^{-p \alpha |w|^{2}(r^{-2}-1)/2} d m_{I_{0}}(w).
\end{split}
\]
Since for all $w\in \mathbb{C}_{I_{0}}$ and $0<r<1$ we have $$e^{-p \alpha |w|^{2}(r^{-2}-1)/2}\le 1,$$ by applying the dominated  convergence theorem we easily obtain $\lim_{r\to 1^{-1}}\|f_r\|^{p}_{p, \alpha, I_{0}}=\|f\|^{p}_{p, \alpha, I_{0}}$.

Therefore, an application of the above mentioned Lemma 3.17 in \cite{HKZ} (see the proof of Theorem 2.1) leads to $\lim_{r\to 1^{-1}}\|f_r - f\|_{p, \alpha, I_{0}}=0$.

{\bf Step 2.} This part is exactly the same as for the case of complex variable, see the proof of part (b) in Proposition 2.9, p. 39  in \cite{Zhu_1}, but reasoning on $\mathbb{C}_{I_{0}}$. Indeed, let $0<r<1$ and choose $r^2\alpha<\beta<\alpha$.  Lemma \ref{GC} allows to see that $f_r\in\mathcal{F}^{\beta,2}_{Slice}(\mathbb{H})$ since $\alpha r^2<\beta$. On the other hand, the condition $\beta<\alpha$ combined with the Lemma \ref{GC} show that $\mathcal{F}^{\beta,2}_{Slice}(\mathbb{H})$ is continuously embedded in $\mathcal{F}^{\alpha, p}_{Slice}(\mathbb{H})$. Moreover, for any $h\in\mathcal{F}^{\beta,2}_{Slice}(\mathbb{H})$ there exists $c=c(p,\alpha,\beta)>0$ such that we have the following estimate
\begin{equation} \label{esth}
\Vert{h}\Vert_{\mathcal{F}^{\alpha, p}_{Slice}(\mathbb{H})}\leq c\Vert{h}\Vert_{\mathcal{F}^{\beta,2}_{Slice}(\mathbb{H})}.
\end{equation}
Note that the family of functions given by $$\displaystyle e_k(q):=\sqrt{\frac{\beta^k}{k!}}q^k,$$ forms an orthonormal basis of $\mathcal{F}^{\beta,2}_{Slice}(\mathbb{H})$ according to \cite{AlpayColomboSabadiniSalomon2014} and $f_r\in\mathcal{F}^{\beta,2}_{Slice}(\mathbb{H})$ for any $0<r<1$. Thus, there exists a sequence $(P_n)_{n\in\mathbb{N}}$ of quaternionic polynomials with right coefficients such that $\Vert{f_r-P_n}\Vert_{\mathcal{F}^{\beta,2}_{Slice}(\mathbb{H})}\longrightarrow 0$ when $n\rightarrow \infty$. Therefore, we just need to use the estimate \eqref{esth} to conclude that the polynomials $(P_n)_{n\in\mathbb{N}}$ approximate $f_r$ in $\mathcal{F}^{\alpha, p}_{Slice}(\mathbb{H})$ for any $0<r<1$.

Note that Proposition 3.1 implies that any other norm (or quasi-norm, depending
on $p$), $\|\cdot \|_{p, \alpha, I}$ with $I\in \mathbb{S}$, is equivalent to the norm (quasi-norm) $\|\cdot \|_{p, \alpha, I_{0}}$. Thus, the  polynomial sequence $(P_{n})_{n\in \mathbb{N}}$ converges to $f$ in any norm (quasi-norm) $\|\cdot \|_{p, \alpha, I}$,
fact which proves the theorem.
\end{proof}
\begin{Rk}{\rm
 The approximation in Fock spaces of the second kind is based on the Taylor expansion since the quaternionic monomials form an orthogonal basis of the Hilbert Fock space of the second kind.}
\end{Rk}
    We now provide a constructive proof for the density result in Theorem \ref{thm_second_kind}, for $1\le p<+\infty$, with quantitative estimates in terms of higher order moduli of smoothness and of the best approximation quantity.

For this end, we introduce the following definition, in which we keep the notations from Section 2.

\begin{Dn}
Let $0<p<+\infty$, $I\in \mathbb{S}$ and $f\in \mathcal{F}^{p}_{\alpha, I}(\mathbb{H})$.
The $L^{p}$-moduli of smoothness of $k$-th order  is defined by
$$\omega_{k}(f ; \delta)_{{\cal{F}}^{p}_{\alpha, I}}=\sup_{0\le |h|\le \delta} \left \{\int_{\mathbb{C}_{I}}|\Delta^{k}_{h}f(z)|^{p}\cdot [e^{-\alpha |z|^{2}/2}]^{p}d m_{I}(z)\right \}^{1/p}=\sup_{0\le |h|\le \delta} \|w_{\alpha} \Delta^{k}_{h}f\|_{L^{p}(\mathbb{C_{I}})},$$
where $k\in \mathbb{N}$, $w_{\alpha}(z)=e^{-\alpha |z|^{2}/2}$,
$$\Delta_{h}^{k}f(z)=\sum_{s=0}^{k}(-1)^{k+s}{k \choose s}f(z e^{I s h}) \mbox{ and }
\|f\|_{L^{p}(\mathbb{C}_{I})}=\left (\int_{\mathbb{C}_{I}}|f(z)|^{p}d m_{I}(z)\right )^{1/p}.$$
(In other words, $\omega_{k}(f ; \delta)_{{\cal{F}}^{p}_{\alpha, I}}=\omega_{k}(f; \delta)_{w_{\alpha}, L^{p}(\mathbb{C}_{I})}$ is a weighted modulus of smoothness.)
\\
The best approximation quantity is defined by
$$E_{n}(f)_{p, \alpha, I}=\inf\{\|f-P\|_{p, \alpha, I} ; P\in {\cal P}_{n}\}=\inf\{\|w_{\alpha}(f-P)\|_{L^{p}(\mathbb{C}_{I})}; P\in {\cal P}_{n}\},$$
where ${\cal P}_{n}$ denotes the set of all polynomials of degree $\le n$.
\end{Dn}
We point out that, as in the case of the $L^{p}$-moduli of smoothness for functions of real variable (see, e.g.,
\cite{DeVore}, pp. 44-45), one may show that $$\lim_{\delta\to 0}\omega_{k}(f ; \delta)_{{\cal{F}}^{p}_{\alpha, I}}=0,$$
\begin{equation}\label{eq6}
\omega_{k}(f ; \lambda\cdot \delta)_{{\cal{F}}^{p}_{\alpha, I}}\le (\lambda+1)^{k}\cdot \omega_{k}(f ; \delta)_{{\cal{F}}^{p}_{\alpha, I}}, \mbox{ if } 1\le p <+\infty
\end{equation}
and
\begin{equation}\label{eq7}
[\omega_{k}(f ; \lambda\cdot \delta)_{{\cal{F}}^{p}_{\alpha, I}}]^{p}\le (\lambda+1)^{k}\cdot [\omega_{k}(f ; \delta)_{{\cal{F}}^{p}_{\alpha, I}}]^{p}, \mbox{ if } 0<p<1.
\end{equation}
In fact, this is a consequence of the fact that setting (for fixed $z$) $g(x)=f(ze^{i x})$, we get $\Delta_{h}^{k}f(z)=\overline{\Delta}_{h}^{k}g(0)$, where $$\overline{\Delta}_{h}^{k}g(x_{0})=\sum_{s=0}^{k}(-1)^{s+k}{k \choose s}g(x_{0}+s h).$$

For any $1\le p<+\infty$ and $f\in \mathcal{F}^{\alpha, p}_{Slice}(\mathbb{H})$, we introduce the convolution operators
$$L_{n}(f)(q)=\int_{-\pi}^{\pi}f(q e^{I_{q} t})\cdot K_{n}(t)d t, \, q\in \mathbb{H},$$
where $K_{n}(t)$ is a positive and even trigonometric polynomial with the property $\int_{-\pi}^{\pi}K_{n}(t)d t=1$.

One can consider the special case of the Fej\'er kernel $$K_{n}(t)=\frac{1}{2 \pi n}\cdot \left (\frac{\sin(n t/2)}{\sin(t/2)}\right )^{2},$$ and, in this case, $L_{n}(f)(q)$ will be denoted by $F_{n}(f)(q)$.

For $$K_{n, r}(t)=\frac{1}{\lambda_{n, r}}\cdot \left (\frac{\sin(n t/2)}{\sin(t/2)}\right )^{2 r},$$ where
$r$ is the smallest integer such that $r\ge \frac{p(m+1)+2}{2}$, $m\in \mathbb{N}$ and the constants $\lambda_{n, r}$ are such that
$\int_{-\pi}^{\pi}K_{n, r}(t)d t=1$, we set
$$I_{n, m, r}(f)(q)=-\int_{-\pi}^{\pi}K_{n, r}(t)\sum_{k=1}^{m+1}(-1)^{k}{m+1 \choose k}f(q e^{I_{q} k t})d t, \, q\in \mathbb{H}.$$
Moreover we define $V_{n}(f)(q)= 2 F_{2 n}(f)(q)-F_{n}(f)(q)$, $q\in \mathbb{H}$.

Following \cite{gasa1}, \cite{gasa2},
for fixed $I\in \mathbb{S}$, if $q\in  \mathbb{C}_{I}$ then $L_{n}(f)(q)$,
$I_{n, m, r}(f)(q)$ and $V_{n}(f)(q)$ are polynomials in $q$ on $\mathbb{C}_{I}$ , with the coefficients independent
of $I$, depending only on the series development of $f$. Consequently, as
functions of $q$, $L_{n}(f)(q)$, $I_{n, m, r}(f)(q)$ and $V_{n}(f)(q)$ are polynomials defined on  $\mathbb{H}$.

We can now state and prove our second main result:
\begin{theorem}\label{Th3.4}
Let $1\le p<+\infty$, $0<\alpha<+\infty$, $m\in \mathbb{N}\bigcup \{0\}$ and $f\in {\cal{F}}^{\alpha, p}_{Slice}(\mathbb{H})$ be arbitrary, fixed.

(i) $I_{n, m, r}(f)(q)$ is a quaternionic polynomial of degree $\le r(n-1)$, satisfying for any $I\in \mathbb{S}$ the estimate
$$\|I_{n, m, r}(f)-f\|_{p, \alpha, I}\le C_{p, m, r}\cdot \omega_{m+1}\left (f; \frac{1}{n}\right )_{{\cal{F}}^{p}_{\alpha, I}}, n\in \mathbb{N},$$
where $m\in \mathbb{N}$, $r$ is the smallest integer with $r\ge \frac{p(m+1)+2}{2}$ and $C(p, m, r)>0$ is a constant independent of $f$, $n$ and $I$.

(ii) $V_{n}(f)(q)$ is a quaternionic polynomial of degree $\le 2n-1$, satisfying for any $I\in \mathbb{S}$ the estimate
$$\|V_{n}(f)-f\|_{p, \alpha, I}\le [2^{(p-1)/p}\cdot (2^{p}+1)^{1/p}+1]\cdot E_{n}(f)_{p, \alpha, I},\, n\in \mathbb{N}.$$
\end{theorem}
\begin{proof}  We refer to \cite{gasa1}, \cite{gasa2} for the fact that the convolution operators $I_{n, m, r}(f)(q)$ and $V_{n}(f)(q)$ are polynomials of the corresponding degrees.

(i) We now apply the following  Jensen type inequality for integrals: if $\int_{-\pi}^{+\pi}G(u)d u=1$, $G(u)\ge 0$ for all $u\in [-\pi, \pi]$ and $\varphi(t)$ is a convex function on the range of the measurable function of real variable $F$, then
$$\varphi\left (\int_{-\pi}^{+\pi}F(u)G(u)d u\right )\le \int_{-\pi}^{+\pi}\varphi(F(u))G(u) du.$$
For any $m\in \mathbb{N}$ we define $r$ to be the smallest integer for which $r\ge \frac{p(m+1)+2}{2}$.

By setting $\varphi(t)=t^{p}$, $1\le p < \infty$, $q\in \mathbb{C}_{I}$, we get
\[
\begin{split}
|f(q)-I_{n, m, r}(f)(q)|^{p}&=\left |\int_{-\pi}^{\pi}\Delta_{t}^{m+1}f(q)K_{n, r}(t)d t\right |^{p}\\
&\le
\left [\int_{-\pi}^{\pi}|\Delta_{t}^{m+1}f(q)|K_{n, r}(t)dt\right ]^{p}\\
&\le \int_{-\pi}^{\pi}|\Delta_{t}^{m+1}f(q)|^{p}K_{n, r}(t)dt.
\end{split}
\]
Multiplying above by $[e^{-\alpha |q|^{2}/2}]^{p}$, integrating on $\mathbb{C}_{I}$ with respect to $d m_{I}(q)$ and making use of the Fubini's theorem, we get
\[
\begin{split}
&\int_{\mathbb{C}_{I}}|I_{n, m, r}(f)(q)-f(q)|^{p}\cdot [e^{-\alpha |q|^{2}/2}]^{p} d m_{I}(q)\\
&\le \int_{-\pi}^{\pi}\left [\int_{\mathbb{C}_{I}}|\Delta_{t}^{m+1}f(q)|^{p}\cdot [e^{-\alpha |q|^{2}/2}]^{p} d m_{I}(q)\right ]K_{n, r}(t)d t\\
&\le \int_{-\pi}^{\pi}\omega_{m+1}(f ; |t|)^{p}_{{\cal{F}}^{p}_{\alpha, I}}\cdot K_{n, r}(t)d t\\
&\le \int_{-\pi}^{\pi}\omega_{m+1}(f ; 1/n)^{p}_{{\cal{F}}^{p}_{\alpha, I}}(n |t|+1)^{(m+1)p}\cdot K_{n, r}(t)d t.
\end{split}
\]
By \cite{Lor}, p. 57, relation (5), for $r\in \mathbb{N}$ with $r\ge \frac{p(m+1)+2}{2}$, we get
\begin{equation}\label{eq8}
\int_{-\pi}^{\pi}(n |t|+1)^{(m+1)p}\cdot K_{n, r}(t)d t\le C_{p, m, r}<+\infty,
\end{equation}
which shows the estimate in (i).

(ii) Now, let $f, g\in \mathcal{F}^{\alpha, p}_{Slice}(\mathbb{H})$ and $1\le p <+\infty$. Since $\varphi(t)=t^{p}$ is convex, for all $a, b \ge 0$ we obviously get the inequality $(a+b)^{p}\le 2^{p-1}(a^{p}+b^{p})$, which for all $q\in \mathbb{C}_{I}$ implies
$$|V_{n}(f)(q)-V_{n}(g)(q)|\le 2|F_{2n}(f)(q)-F_{2n}(g)(q)|+|F_{n}(f)(q)-F_{n}(g)(q)|$$
$$\le 2\int_{-\pi}^{\pi}|f(q e^{I t})-g(q e^{I t})|\cdot K_{2n}(t)d t+\int_{-\pi}^{\pi}|f(q e^{I t})-g(q e^{I t})|\cdot K_{n}(t)d t$$
and
$$|V_{n}(f)(q)-V_{n}(g)(q)|^{p}\le 2^{p-1}\left [\left (2\int_{-\pi}^{\pi}|f(q e^{I t})-g(q e^{I t})|\cdot K_{2n}(t)d t\right )^{p}\right .$$
$$+\left .\left (\int_{-\pi}^{\pi}|f(q e^{I t})-g(q e^{I t})|\cdot K_{n}(t)d t\right )^{p}\right ]$$
$$\le 2^{p-1}\left [2^{p}\int_{-\pi}^{\pi}|f(q e^{I t})-g(q e^{I t})|^{p}\cdot K_{2n}(t)d t +
\int_{-\pi}^{\pi}|f(q e^{I t})-g(q e^{I t})|^{p}\cdot K_{n}(t)d t\right ].$$
Multiplying above with $[e^{-\alpha|q|^{2}/2}]^{p}=[e^{-\alpha|q e^{I t}|^{2}/2}]^{p}$,
integrating this inequality with respect to $d m_{I}(q)$ on $\mathbb{C}_{I}$ and reasoning as at the above point (i), we get
$$\|V_{n}(f)-V_{n}(g)\|^{p}_{p, \alpha, I}\le 2^{p-1}\left [2^{p}\int_{-\pi}^{\pi}\left (\int_{\mathbb{C}_{I}}|f(q e^{I t})-g(q e^{I t})|^{p}[e^{-\alpha|q e^{I t}|^{2}/2}]^{p}d m_{I}(q)\right )K_{2n}(t)d t +\right .$$
$$\left .\int_{-\pi}^{\pi}\left (\int_{\mathbb{C}_{I}}|f(q e^{I t})-g(q e^{I t})|^{p}[e^{-\alpha|q e^{I t}|^{2}/2}]^{p}d m_{I}(q)\right )K_{n}(t)d t\right ].$$
 Setting $F(q)=|f(q)-g(q)|^{p}[e^{-\alpha|q|^{2}/2}]^{p}$, $q\in \mathbb{C}_{I}$, writing $q=r\cos(\theta)+ I r\sin(\theta)$ and taking into account that
$$
d m_{I}(q)=\frac{1}{\pi}r d r d\theta,
$$
with easy calculations we get the equality
$$\int_{\mathbb{C}_{I}}|F(q e^{I t})|^{p}d m_{I}(q)=\int_{\mathbb{C}_{I}}|F(q)|^{p}d m_{I}(z),\mbox{ for all } t,$$
which replaced in the above inequality easily implies
$$\|V_{n}(f)-V_{n}(g)\|^{p}_{p, \alpha, I}\le 2^{p-1}[2^{p}\|f-g\|^{p}_{p, \alpha}+\|f-g\|^{p}_{p, \alpha}]=2^{p-1}(2^{p}+1)\|f-g\|^{p}_{p, \alpha},$$
that is
$$\|V_{n}(f)-V_{n}(g)\|_{p, \alpha, I}\le 2^{(p-1)/p}\cdot (2^{p}+1)^{1/p}\|f-g\|_{p, \alpha, I}.$$
Now, let us denote by $P_{n}^{*}$ a polynomial of best approximation by elements in ${\cal P}_{n}$ in the norm $\|\cdot \|_{p, \alpha, I}$, that is
$$E_{n}(f)_{p, \alpha, I}=\inf\{\|f-P\|_{p, \alpha, I} ; P\in {\cal P}_{n}\}=\|f-P_{n}^{*}\|_{p, \alpha, I}.$$
Being ${\cal P}_{n}$ finite dimensional for any fixed $n$, the polynomial $P_{n}^{*}$ exists.

 Mimicking the reasonings  in \cite{gal2}, p. 425 we get $V_{n}(P^{*}_{n})=P^{*}_{n}$, for all $q\in \mathbb{C}_{I}$, so that we have
\[
\begin{split}
\|f-V_{n}(f)\|_{p, \alpha, I}&\le \|f-P_{n}^{*}\|_{p, \alpha, I}+\|V_{n}(P_{n}^{*})-V_{n}(f)\|_{p, \alpha, I}\\
&\le E_{n}(f)_{p, \alpha, I}+2^{(p-1)/p}\cdot (2^{p}+1)^{1/p}\|P_{n}^{*}-f\|_{p, \alpha, I}\\
&=[2^{(p-1)/p}\cdot (2^{p}+1)^{1/p}+1]\cdot E_{n}(f)_{p, \alpha, I},
\end{split}
\]
which proves (ii) and the theorem.
\end{proof}
\begin{Rk} {\rm The result in Theorem \ref{Th3.4} evidently holds also in the complex Fock spaces. In this context, the result is new.}
\end{Rk}

In \cite{AlpayColomboSabadiniSalomon2014} the authors proved that the quaternionic Fock space of the second kind $\mathcal{F}^{\alpha,2}_{Slice}(\mathbb{H})$ is a right quaternionic Hilbert space whose reproducing kernel is given for all $(r,q)\in\mathbb{H}^2$ by \[
\begin{split}
K_\alpha(r,q)&:= e_*(\alpha r \overline{q})\\
&=\displaystyle \sum_{k=0}^\infty \frac{\alpha^k r^k\overline{q}^k}{k!} .
\end{split}
\]

Then, we denote by $\mathcal{R}$ the set of all functions of the form $$f(r)=\displaystyle \sum_{k=1}^nK_\alpha(r,q_k)b_k, \forall r\in\mathbb{H} $$ where $(b_k)_k, (q_k)_k\in\mathbb{H}$ for all $k=1,..,n.$ As a consequence of the Theorem \ref{thm_second_kind} we obtain the following result:
\begin{theorem}\label{RK2} Let $\alpha > 0$ and $0<p<\infty$. The set $\mathcal{R}$ is dense in the quaternionic Fock spaces of the second kind $\mathcal{F}^{\alpha, p}_{Slice}(\mathbb{H})$.
\end{theorem}
\begin{proof}
\begin{itemize}
\item[i)] The result is clear in the Hilbert case when $p=2$. Indeed, we only use the reproducing kernel property to see that the orthogonal of $\mathcal{R}$  is reduced to zero.
\item[ii)] For $p>0$, let $f$ be a quaternionic polynomial with right coefficients. Then, there exists $0<\beta<\alpha$ such that $\mathcal{F}^{\beta, 2}_{Slice}(\mathbb{H})$ is continuously embedded in $\mathcal{F}^{\alpha, p}_{Slice}(\mathbb{H})$. Note that since $f$ is  a polynomial, by the Hilbert case it can be approximated by a sequence  of $\mathcal{R}$ in the topology norm of  $\mathcal{F}^{\beta, 2}_{Slice}(\mathbb{H})$. Thus, let $q_1,...,q_n\in\mathbb{H}$ and $(a_k)_{k=1,...,n}\subset\mathbb{H}$ be such that $\displaystyle\|f-\sum_{k=1}^n K_\beta^{\frac{\alpha}{\beta}q_k}a_k,\|_{\mathcal{F}^{\beta,2}_{Slice}(\mathbb{H})}$ tends to zero as $n\rightarrow\infty$. However, there exists $c>0$ such that we have the following estimate

\[
\begin{split}
\|f-\sum_{k=1}^n K_\alpha^{q_k}a_k,\|_{\mathcal{F}^{\alpha,p}_{Slice}(\mathbb{H})}&\le c\|f-\sum_{k=1}^n K_\alpha^{q_k}a_k,\|_{\mathcal{F}^{\beta,2}_{Slice}(\mathbb{H})}\\
&\le
c\|f-\sum_{k=1}^n K_\beta^{\frac{\alpha}{\beta}q_k}a_k,\|_{\mathcal{F}^{\beta,2}_{Slice}(\mathbb{H})}.
\end{split}
\] This shows that $\displaystyle\|f-\sum_{k=1}^n K_\alpha^{q_k}a_k,\|_{\mathcal{F}^{\alpha,p}_{Slice}(\mathbb{H})}$ tends to zero as $n\rightarrow\infty$, for any quaternionic polynomial $f$. However, in Theorem \ref{thm_second_kind} we proved that the set of quaternionic polynomials is dense in any quaternionic Fock space of the second kind. Hence, $\mathcal{R}$ is dense in the quaternionic Fock spaces of the second kind $\mathcal{F}^{p}_{\alpha}(\mathbb{H})$.
\end{itemize}

\end{proof}
The order and type of slice regular entire functions on quaternions were introduced
in Chapter 5 of the book \cite{ColomboSabadiniStruppa2016}. In the setting of  the Fock spaces $\mathcal{F}^{\alpha, p}_{Slice}(\mathbb{H})$, we have:
\begin{proposition}
Let $0<p< \infty$ and $f\in\mathcal{F}^{\alpha, p}_{Slice}(\mathbb{H})$. Then, $f$ is of order less or equal than $2$. Moreover, if $f$ is of order $2$, then it is of type $\sigma(f)\leq \frac{\alpha}{2}.$
\end{proposition}
\begin{proof}
Note that $f\in\mathcal{F}^{\alpha, p}_{Slice}(\mathbb{H})$, then by Lemma \ref{GC}, we have $$\vert f(q)\vert \leq ce^{\frac{\alpha}{2}\vert{q}\vert^2}\Vert f \Vert_{\mathcal{F}^{\alpha, p}_{Slice}(\mathbb{H})}.$$ In particular, we have $$M_f(r)=\underset {\vert q\vert=r} \max \vert f(q) \vert \leq ce^{\frac{\alpha}{2}r^2}\Vert f \Vert_{\mathcal{F}^{\alpha, p}_{Slice}(\mathbb{H})}.$$

Therefore, $$\rho(f)=\lim_{r\rightarrow \infty}\frac{\log (\log M_f(r))}{\log r} \leq 2.$$

Moreover, if $\rho(f)=2,$ then we have $$\sigma(f)=\lim_{r\rightarrow \infty}\frac{\log M_f(r)}{r^2} \leq \frac{\alpha}{2}.$$
\end{proof}

\noindent{\bf Acknowledgements} \\ \\
Kamal Diki acknowleges the support of the project “INdAM Doctoral Programme in Mathematics and/or Applications Cofunded by Marie Sklodowska-Curie Actions”, acronym: INdAM-DP-COFUND-2015, grant number: 713485.

\end{document}